\documentclass[letterpaper,10pt,reqno,final]{amsart}
\usepackage{amsmath}%
\usepackage{amsfonts}%
\usepackage{amssymb}%
\usepackage{graphicx}
\usepackage{functan}
\usepackage{mathrsfs}
\usepackage[all,cmtip]{xy}
\usepackage[obeyspaces]{url}
\usepackage{braket}
\newtheorem{theorem}{Theorem}

\newtheorem{corollary}{Corollary}

\newtheorem{definition}{Definition}
\newtheorem{example}{Example}

\newtheorem{lemma}{Lemma}

\newtheorem{remark}{Remark}

\numberwithin{equation}{section}
\numberwithin{theorem}{section}
\numberwithin{lemma}{section}
\numberwithin{corollary}{section}
\numberwithin{definition}{section}
\numberwithin{example}{section}
\numberwithin{remark}{section}
\begin{document}
\title[Approximation of Schr{\"o}dinger Unitary Groups]{Approximation of Schr{\"o}dinger Unitary Groups of Operators by Particular Projection Methods}
\author{Fredy Vides}
\address
{Escuela de Matem{\'a}tica y Ciencias de la Computaci{\'o}n \newline%
\indent Universidad Nacional Aut{\'o}noma de Honduras}%
\email{fvides@unah.edu.hn}
\urladdr{http://fredyvides.6te.net}
\thanks{This research has been performed in part thanks to the financial support of the School of Mathematics and Computer Science of the National University of Honduras.}
\date{\today}

\begin{abstract}
In this paper we work with the approximation of unitary groups of operators of the form $e^{-itH}$ where $H\in\mathscr{L}(\mathcal{H})$ is the self-adjoint Hamiltonian of a given Hermitian quantum dynamical system modeled in the discretizable
Hilbert space $\mathcal{H}=\mathcal{H}(G)$, to perform such approximations we implement some techniques from operator theory that we name particular projection methods by compatibility  with quantum theory conventions. Once particular representations are defined we study the interelation between some of them properties with the original operators that they mimic. In the end some estimates for numerical implementation are presented to verify the theoretical discussion.
\end{abstract}
\maketitle

\section{Introduction} \label{intro}
In this work we will focus our attention in the approximation of Schr{\"o}dinger semigroups of operators that will in be described in general by the set $\{U_t:=e^{-itH}:t\in R\}$, whose elements clearly satisfy the semigroups conditions: (i) $U_0=\mathbf{1}$, (ii) $U_t\circ U_s(\cdot)=U_{t+s}(\cdot)$ and (iii) $\lim_{h\to0^+}U_hx=x, \forall x\in \mathcal{D}\subseteq \mathcal{H}(G)$, besides by theorem T.\ref{stone} the condtion (iv) $\norm{}{U_t u_0}=\norm{}{u_0}, t\in\mathbf{R}$, will be also satisfied when $H$ is self adjoint, wich means that $U_t$ is unitary for any $t\in\mathbf{R}$. In the expression presented above the operator $H\in\mathscr{L}(\mathcal{H})$ is the Hamiltonian of a given quantum dynamical system whose abstract evolution equation
is given by:
\begin{equation}
E \psi(t)=H\psi(t)
\label{abs_ev}
\end{equation}
with $\psi(0)=\psi_0\in\mathcal{H}$ and where $E\longrightarrow\frac{i}{\hbar}\partial_t$, here $H\in\mathscr{L}(\mathcal{H})$ will in general have the form $H=p^\dagger p+V(\cdot)$ with $p^\dagger\longrightarrow -\frac{i}{\hbar}\nabla+b$ and with $V\in C^{\alpha=1}(\mathcal{B}\subset\mathcal{H})$, for simplicity, in this work we will consider our scale such that $\hbar=1$, also we will have that in some suitable sense the operator $H\in\mathscr{L}(\mathcal{H})$ will be restricted by some boundary conditions related to the media where a particular quantum dynamical system evolves.

In the following sections we will implement some operator theory techniques in the theoretical analysis of the approximation schemes of the Schr{\"o}dinger semigroups and in the end some numerical implementations will be presented.

\section{Basics of Quantum Dynamical Systems} \label{qds}

Quantum dynamical systems can be studied using several types of operators, in this work we will consider in general quantum dynamical systems that evolve in a particular discretizable space of states $\mathcal{H}(G)$, i.e., a separable reproducing kernel Hilbert Space with $G\subset\subset \mathbf{R}^N$. A vector $\psi(t)\in\mathcal{H}(G)$ that satisfies \eqref{abs_ev} receives the name of wave function, the evolution of a given quatum system can be described using the wave function, the time evolution of the wave function can be computed using the corresponding Schr{\"o}dinger semigroup related to a particular quantum dynamical system using the following expression
\begin{equation}
\psi(t)=e^{-itH}\psi_0.
\label{ev2}
\end{equation}
The vector $\psi(t)\in\mathcal{H}(G)$ can also be considered like a probability amplitud related to measurements concerning to the position of the particles in a quantum system modeled by an abstract Shr{\"o}dinger evolution equation like \eqref{abs_ev}. Usually we can express quantum evolution equations using the Dirac's braket notation, we usually have that the \textit{ket} operation $\ket{\cdot}:\mathcal{H}\rightarrow\mathcal{H}$, is defined explicitly by
\begin{equation}
\ket{\phi}:=\phi, \phi\in\mathcal{H},
\end{equation}
on the other hand we have that the \textit{bra} operation $\bra{\cdot}[\cdot]:\mathcal{H}\times\mathcal{H}\rightarrow\mathbf{C}$, is defined by
\begin{equation}
\bra{\phi}[\cdot]:=\scalprod{}{\cdot}{\phi}, \phi\in\mathcal{H},
\end{equation}
it can be seen that the natural inner product of the space of states $\mathcal{H}$ can be expressed using the Dirac's braket notation in the form
\begin{equation}
\scalprod{}{\xi}{\phi}=\bra{\phi}[\xi]=\bra{\phi}\ket{\xi}=\braket{\phi|\xi}, \xi,\phi\in\mathcal{H},
\end{equation}
for any quantum operator $A\in\mathscr{L}(\mathcal{H})$ and any pair $\xi,\phi\in\mathcal{H}$, the operation $\scalprod{}{A\xi}{\phi}$ can  be expressed explicitly by
\begin{equation}
\scalprod{}{A\xi}{\phi}=\bra{\phi}[A\xi]=\bra{\phi}A\ket{\xi}=\braket{\phi|A\xi}, \xi,\phi\in\mathcal{H}.
\end{equation}
The probability density for a specific time will be given by $|\psi(t)|^2=\psi(t)\overline{\psi}(t)$, this statistical approach of the wave function, and the corresponding inner product of $\mathcal{H}(G)$ alows us to compute the expectation $\left\langle B \right\rangle_t\in \mathbf{R}$ of a given observable $B$, i.e. a quantum operator $B\in\mathscr{L}(\mathcal{H})$, using the following expression:
\begin{equation}
\braket{B}_t:=\frac{\scalprod*{\mathcal{H}(G)}{B\psi(t)}{\psi(t)}}{\scalprod*{\mathcal{H}(G)}{\psi(t)}{\psi(t)}}
\label{expc_op}
\end{equation}
wich in Dirac's braket notation is equivalent to
\begin{equation}
\braket{B}_t=\frac{\bra{\psi(t)}B\ket{\psi(t)}}{\braket{\psi(t)|\psi(t)}}.
\label{expc_op2}
\end{equation}
For any two operators $A,B\in\mathscr{L}(\mathcal{H}(G))$ on discretizable quantum space of states $\mathcal{H}(G)$ we can define a commutator operation by
\begin{equation}
[A,B]:=AB-BA
\end{equation}
we will have that two operators $X,Y\in\mathscr{L}(\mathcal{H}(G))$ commute if, and only if $[X,Y]=\mathbf{0}$, when two operators commute we also say that they are compatible. In general the operation $[\cdot,\cdot]:\mathscr{L}(\mathcal{H})\times \mathscr{L}(\mathcal{H})\rightarrow \mathscr{\mathcal{H}}$ defines an operator in $\mathscr{L}(\mathcal{H})$. If we obtain a nomalized representation $\ket{\Psi(t)}:=\ket{\psi(t)}/\norm{}{\psi(t)}$ of the wave function $\psi(t)\in\mathcal{H}$, then \eqref{expc_op2} can be represented by
\begin{equation}
\braket{B}_t=\bra{\Psi(t)}B\ket{\Psi(t)}.
\end{equation}

\begin{theorem} \label{const_motion} Constants of Motion. If a quantum operator $A\in\mathscr{L}(\mathcal{H})$ is constant in time and compatible with the Hamiltonian $H\in\mathscr{L}(\mathcal{H})$ in \eqref{abs_ev}, and if $H\in\mathscr{L}(\mathcal{H})$ is symmetric then $\braket{A}_t$ defines a constant of motion.
\end{theorem}
\begin{proof}
Since $H\in\mathscr{L}(\mathcal{H})$ si symmetric, it can be seen that
\begin{eqnarray*}
\frac{d}{dt}\braket{A}_t&=&i\bra{\Psi(t)}HA\ket{\Psi(t)}-i\bra{\Psi(t)}AH\ket{\Psi(t)}\\
&=&i\bra{\Psi(t)}[H,A]\ket{\Psi(t)}
\end{eqnarray*}
since we also have that $A$ and $H$ are compatible, then $[H,A]=\bf{0}$ and $\frac{d}{dt}\braket{A}_t=0$, therefore $\braket{A}_t$ is a constant of motion.
\end{proof}

\section{Particular Projection Methods}
In this section we will describe the approximation techniques implemented for spatial discretization of the operators
related to the dynamics and physical measurements of the quantum systems described here.

\subsection{Particular Projection in $\mathcal{H}(G)$} 

In this section and some other below we will work with the spatial discretization of operators that are present in a quantum dynamical system in Schr{\"o}-dinger picture. When we want to build a discretization of a given spatial operator we first need to define a grid, wich is set $G_{m,h}\subset G\subset \mathbf{R}^N$ that depends in some suitable sense of the parameters $m\in \mathbf{Z}^+$ and $h\in\mathbf{R}^+$, in particular the  cardinality of the grid denoted by $N_{m,h}:=|G_{m,h}|$ depends on $m,h$ throug the following rules $N_{m,h}\geq N_{m',h}, m\geq m'$ and $N_{m,h}\geq N_{m,h'}, h\leq h'$. 

Once we have defined a grid $G_{m,h}\subset G$ on the media where a quantum dynamical system evolves, we can define a particular projection $P_\mathcal{S}\in\mathscr{L}(\mathcal{H}(G),\mathcal{H}_\mathcal{S}(G))$ for $\mathcal{H}_\mathcal{S}\leqslant \mathcal{H}$, with respect to this grid, wich is a projection that can be factored in the form $P_\mathcal{S}:=p_\mathcal{S}p^\dagger_\mathcal{S}$, where $p_\mathcal{S}\in\mathscr{L}(\mathcal{H}'_\mathcal{S},\mathcal{H}_\mathcal{S})$ and $p^\dagger_\mathcal{S}\in\mathscr{L}(\mathcal{H},\mathcal{H}'_\mathcal{S})$ are particular summation and decompostion operators respectively. A particular decomposition operator $p^\dagger_\mathcal{S}\in\mathscr{L}(\mathcal{H}(G),\mathcal{H}_\mathcal{S}(G))$ is related to a given grid $G_{m,h}\subset G$ through the following explicit definition
\begin{equation}
p^\dagger_\mathcal{S}v:=\{c_k(v,G_{m,h})\}\in\mathbf{C}^{N_{m,h}}, v\in\mathcal{H}(G)
\end{equation}
the particular summation operator is related to a given set $\mathcal{H}(G)\supset\mathcal{B}\supset\mathcal{B}_m:={b_k}$ that is a subset of a particular basis $\mathcal{B}\subset \mathcal{H}(G)$, whose elements are compatible in some suitable sense with boundary conditions of wave functions in the physical media, and that satisfies $|\mathcal{B}_m|=|G_{m,h}|$, thorugh the following expression
\begin{equation}
p_\mathcal{S}x:=\sum_k x_kb_k, x=\{x_k\}\in\mathbf{C}^{N_{m,h}}
\end{equation}
the functionals $c_k(\cdot,G_{m,h})$ used for the definition of particular decomposition operators are considered in general to satisfy that for any $b\in\mathcal{B}_m$, $p_\mathcal{S}p^\dagger_\mathcal{S}b=b$, sometimes we also impose the condition $c_k(b_j,G_{m,h})=\alpha_j\delta_{k,j}$, with $\delta_{i,j}$ the Kronecker delta, the above condition is called pseudo-orthogonality or pseudo-orthonormality when $\alpha_j=1,\forall j$. For a given particular projection $P_\mathcal{S}\in\mathscr{L}(\mathcal{H},\mathcal{H}_\mathcal{S})$ we say that it has approximation order $\nu_m$, where $\nu_m$ is a number that depends in some sense on the grid paremeter $m$, if we have that for any $v\in\mathcal{H}(G)$ the projection satisfies the relation
\begin{equation}
\norm{}{P_\mathcal{S}v-v}\leq c_v h^{\nu_m}
\end{equation}
with respect a prescribed norm $\norm{}{\cdot}$ in $\mathcal{H}(G)$ and where $h$ is the mesh size of the prescribed grid. Sometimes we represent a paricular projection $P_\mathcal{S}\in\mathscr{L}(\mathcal{H},\mathcal{H}_\mathcal{S})$ in an alternative form given by $P_{m,h}$ where $m$ and $h$ are the grid parameter and mesh size respectively.

\subsection{Inner Product Matrices} For a given discretizable Hilbert space $\mathcal{H}(G)$ whose inner product is induced by the inner product map $\mathcal{M}:\mathcal{H}\times\mathcal{H}\rightarrow \mathbf{C}$ in the form
\begin{equation}
\scalprod*{\mathcal{H}}{u}{v}:=\mathcal{M}[u](v)
\end{equation} 
we can obtain a discrete representation $\mathbb{M}_\mathcal{S}$ of the inner product map $\mathcal{M}$ with respect to a particular projection $P_\mathcal{S}\in\mathscr{L}(H,H_\mathcal{S})$, through the follwing explicit definition
\begin{equation}
(\cdot)^\ast\mathbb{M}_\mathcal{S}[\cdot]:=\mathcal{M}[p_\mathcal{S}\cdot](p_\mathcal{S}\cdot)
\end{equation}
the matrix $\mathbb{M}_\mathcal{S}$ receives the name of inner product matrix. The followng result was proved in \cite{Vides}.

\begin{theorem}
Every inner product matrix is symmetric positive definite.
\end{theorem}

If for any $u,v\in\mathcal{H}$ we take $\mathbf{u}:=p^\dagger_\mathcal{S} u$ and $\mathbf{v}:=p^\dagger_\mathcal{S} v$, it can be seen that we can express the operation $\scalprod{}{P_\mathcal{S}u}{P_\mathcal{S}v}$ using the inner product matrix in the following way
\begin{equation}
\scalprod{}{P_\mathcal{S}u}{P_\mathcal{S}v}:=\mathcal{M}[p_\mathcal{S}p^\dagger_\mathcal{S}u](p_\mathcal{S}p^\dagger_\mathcal{S}v)=\mathbf{v}^\ast\mathbb{M}_\mathcal{S}\mathbf{u}
\end{equation}
since $\mathbb{M}_\mathcal{S}$ is symmetric positive definite, and if we denote by $\mathbb{W}_\mathcal{S}$ the formal square root of $\mathbb{M}_\mathcal{S}$, we can obtain an alternative expression for the above operation that will be given by
\begin{equation}
\scalprod{}{P_\mathcal{S}u}{P_\mathcal{S}v}=(\mathbb{W}_\mathcal{S}\mathbf{v})^\ast(\mathbb{W}_\mathcal{S}\mathbf{u})=\scalprod*{2}{\mathbb{W}_\mathcal{S}\mathbf{u}}{\mathbb{W}_\mathcal{S}\mathbf{v}}
\end{equation}
where $\scalprod*{2}{\cdot}{\cdot}$ is the complex euclidian inner product. In a similar way one can express the operation $\norm{}{P_\mathcal{S}u}$, where $\norm{}{\cdot}$ is the norm induced by $\scalprod{}{\cdot}{\cdot}$ in $\mathcal{H}$ through the following relations
\begin{equation}
\norm{}{P_\mathcal{S}u}:=(\scalprod{}{P_\mathcal{S}u}{P_\mathcal{S}u})^{1/2}=(\scalprod*{2}{\mathbb{W}_\mathcal{S}\mathbf{u}}{\mathbb{W}_\mathcal{S}\mathbf{u}})^{1/2}=\norm*{2}{\mathbb{W}_\mathcal{S}\mathbf{u}}.
\end{equation}

We can express kets and bras in a discrete frame with respect to a particular projection $P_\mathcal{S}\in\mathscr{L}(\mathcal{H})$, using the rules
\begin{eqnarray}
&\ket{\psi}&\conv*{p^\dagger_\mathcal{S}}{}\ket{\Psi}:=\Psi\\
&\bra{\psi}&\conv*{(p^\dagger_\mathcal{S})^\ast}{}\bra{\Psi}:=\Psi^\ast \mathbb{M}
\end{eqnarray}
for ket and bra operations respectively.
\subsection{Particular Representation of operators} We can obtain discrete representations of operators in $\mathscr{L}(\mathcal{H},\mathcal{H}')$ for $\mathcal{H},\mathcal{H}'$ dicretizable Hilbert spaces, with respect to particular projections $P_\mathcal{S}\in\mathscr{L}(\mathcal{H})$ and $Q_\mathcal{V}\in\mathscr{L}(\mathcal{H}')$, the corresponding discretization will receive the name of particular representation, for any given $A\in\mathscr{L}(\mathcal{H},\mathcal{H}')$ we denote its particular representation with respect to $P_\mathcal{S}$ and $Q_\mathcal{V}$ by $\mathbb{A}\in\mathbf{C}^{N_\mathcal{V}\times N_\mathcal{S}}$ and we define it explicitly in the following way
\begin{equation}
\mathbb{A}:=q^\dagger_\mathcal{V}Ap_\mathcal{S}
\end{equation}
A particular representation $\mathbb{A}\in\mathbf{C}^{N_\mathcal{V}\times N_\mathcal{S}}$ is said to have approximation order $\mu_m$, with $\mu_m$ a number that depends on a given grid parameter $m$ related to the particular projections, if it satisfies the relation
\begin{equation}
\norm{}{Q_\mathcal{V}AP_\mathcal{S}u-Q_\mathcal{V}Au}=\norm*{2}{\mathbb{W}_\mathcal{V}(\mathbb{A}\mathbf{u}-q_\mathcal{V}^\dagger Au)}\leq c_u h^{\mu_m}
\end{equation}
with respect to a prescribed norm $\norm{}{\cdot}$ in $\mathcal{H}'$, where $h$ is the mesh size of the prescribed grid, an alternative expression for this property can be obtained when $\mathcal{H}=\mathcal{H}'$ and $P_\mathcal{S}=Q_\mathcal{V}$, taking $c_A:=\sup_u c_u$, for $c_u$ in the above equation, and then writing:
\begin{equation}
\norm{}{[P_\mathcal{S},A]}\leq c_A h^{\mu_m}
\end{equation}
when $h\to 0^+$ we say that the pair $P_\mathcal{S},A\in\mathscr{L}(\mathcal{H})$ almost commute.

\begin{remark} \label{obs1}
It is important to note that for any given $B\in\mathscr{L}(\mathcal{H})$ and any particular projector $P_\mathcal{S}\in\mathscr{L}(\mathcal{H},\mathcal{H}_\mathcal{S})$, we will have that $P_\mathcal{S}\phi=\phi, \forall\phi\in\mathcal{H}_\mathcal{S}:=P_\mathcal{S}\mathcal{H}$ and $B\phi=P_\mathcal{S}B\phi, \forall \phi\in \mathcal{H}_\mathcal{S}$.
\end{remark}

\subsection{Exactly Factorizable Operators} Given two Hilbert spaces $X,Y$, an operator $A:X\longrightarrow X$ is said to exactly factorizable, if it can be writen in the form $A:=BC$, with $C:X\longrightarrow Y$ and $B:Y\longrightarrow X$, in this article we will focus our attention on exactly factorizable operators of the form $A:=BC$, with $B^\dagger:=\alpha C$, for $\alpha=\pm1\in\mathbf{R}\backslash\{0\}$, these conditions imply that
\begin{equation}
\scalprod*{X}{Au}{v}=\scalprod*{Y}{Cu}{B^\dagger v}=\alpha\scalprod*{Y}{Cu}{Cv}
\label{weak_fact_cond}
\end{equation}
wich permits us to obtain the following relation
\begin{equation}
\scalprod*{X}{Au}{v}=\alpha\scalprod*{Y}{Cu}{Cv}=\scalprod*{Y}{u}{\alpha C^\dagger Cv}=\scalprod*{X}{u}{B Cv}=\scalprod*{X}{u}{Av}
\label{weak_factors}
\end{equation}
wich implies that exactly factorizable operators of this type are self-adjoint. Now if we take two finite rank particular projections $P_\mathcal{S}\in\mathscr{L}(X,X_\mathcal{S})$ and $Q_\mathcal{V}\in\mathscr{L}(Y,Y_\mathcal{S})$, and if we define the particular representations $\mathbb{A}:=p^\dagger_\mathcal{S}Bp_\mathcal{S}$, $\mathbb{B}:=p^\dagger_\mathcal{S}Bq_\mathcal{V}$ and $\mathbb{C}:=q^\dagger_\mathcal{V}Cp_\mathcal{S}$ for $A$, $B$ and $C$ respectively, we can first note that

\begin{equation}
V^\ast\mathbb{M}_\mathcal{V}\mathbb{C}U=\scalprod*{\mathcal{V}}{\mathbb{C}U}{V}=\scalprod*{\mathcal{S}}{U}{\mathbb{C}^\dagger V}=V^\ast(\mathbb{C}^\dagger)^\ast\mathbb{M}_\mathcal{S}U
\end{equation}

wich implies that $\mathbb{C}^\dagger:=\mathbb{M}_\mathcal{S}^{-1}\mathbb{C}^\ast\mathbb{M}_\mathcal{V}$, then from the definition of particular projections we can obtain the following relations

\begin{eqnarray}
V^\ast\mathbb{M}_\mathcal{S}\mathbb{A}U&=&\scalprod*{\mathcal{S}}{\mathbb{A} U}{V}\\
&=&\alpha\scalprod*{\mathcal{V}}{\mathbb{C}U}{\mathbb{C}V}\\
&=&\alpha(\mathbb{C}V)^\ast\mathbb{M}_\mathcal{V}\mathbb{C}U\\
&=&V^\ast\mathbb{C}^\ast\mathbb{M}_\mathcal{V}[\alpha\mathbb{C}]U.
\label{weak_factors_disc}
\end{eqnarray}

The above expressions permit us to represent $\mathbb{A}$ by $\mathbb{A}:=\mathbb{M}_\mathcal{S}^{-1}\mathbb{C}^\ast\mathbb{M}_\mathcal{V}[\alpha\mathbb{C}]=\alpha\mathbb{C}^\dagger\mathbb{C}$, and from this we will have that the particular representation of $A$ preserves self-adjointness and operator sign (positive/negative) according to $\alpha\in\mathbf{R}\backslash\{0\}$.

\section{Discrete Time Integration and Matrix Schr{\"o}dinger Unitary Groups} 

As we discussed in \textsection \ref{intro}, we will consider that the Hamiltonians $H\in\mathscr{L}(\mathcal{H}(G))$ in Scr{\"o}dinger models like \eqref{abs_ev}, have the form $H=H_0+V(\cdot)$, where $H_0$ defined by
\begin{equation}
H_0:=p^\dagger p
\end{equation}
and with $p\in\mathscr{L}(\mathcal{H},\mathcal{H}')$ and $p^\dagger\in\mathscr{L}(\mathcal{H}',\mathcal{H})$, we will consider that $V\in C^{\alpha=1}(\mathcal{D})$, $\mathcal{D}\subseteq\mathcal{H}\cap\mathcal{H}'$. The abstract initial value problem
\begin{equation}
\left \{
\begin{array}{l}
\ket{\psi'(t)}=-iH_0\ket{\psi(t)}-iV(\ket{\psi(t)})\\
\ket{\psi(0)}=\ket{\psi_0}
\end{array}
\right .
\label{ivp}
\end{equation}
\medskip
derived from \eqref{abs_ev}, can be rewrited introducing the integrating factor $e^{itH_0}$, in the form
\begin{equation}
\ket{\psi(t)}=e^{-itH_0}\ket{\psi_0}+\int_0^t e^{i(s-t)H_0}V(\ket{\psi(t)})ds
\end{equation}
it can be seen that $H_0$ is exactly factorizable, hence, self-adjoint and by theorem T.\ref{stone} we will have that $e^{\pm itH_0}, t\in\mathbf{R}$ will be unitary, since we also have that $V\in C^{\alpha=1}(\mathcal{D})$, we can derive the following result.
\begin{theorem}
The operator $U\in\mathscr{B}(\mathcal{D})$ defined by
\begin{equation}
U(\phi):=e^{-itH_0}\ket{\psi_0}+\int_0^t e^{i(s-t)H_0}V(\ket{\phi(t)})ds
\label{exp_contract}
\end{equation}
is a strict contraction with respect to the norm $\norm*{C([-T,T],\mathcal{D}(G))}{\cdot}$ defined by
\begin{equation}
\norm*{C([-T,T],\mathcal{D}(G))}{\xi}:=\sup_{t\in[-T,T]}\norm*{\mathcal{D}(G)}{\xi}, \xi\in\mathcal{D}(G).
\end{equation}
when $T<1/c_V.$
\end{theorem}
\begin{proof}
Since for any $t\in\mathbf{R}$ $e^{\pm itH_0}$ is unitary and since $V\in C^{\alpha=1}(\mathcal{D})$ we will have that
\begin{eqnarray*}
&&\norm*{C([-T,T],\mathcal{D}(G))}{U(\ket{\phi})-U(\ket{\omega})}=\\ &&\norm*{C([-T,T],\mathcal{D}(G))}{\int_0^te^{i(s-t)H_0}[V(\ket{\phi(t)})-V(\ket{\omega(t)})]ds}\leq\\
&&c_V\norm*{C([-T,T],\mathcal{D}(G))}{\ket{\phi(t)}-\ket{\omega(t)}}\int_0^t ds \leq\\
&&c_VT\norm*{C([-T,T],\mathcal{D}(G))}{\ket{\phi(t)}-\ket{\omega(t)}}
\end{eqnarray*}
taking $T<1/c_V$ the result follows.
\end{proof}
\begin{corollary} There exists a unique solution for \eqref{ivp} in a time interval $[-T,T]$ $\subseteq\mathbf{R}$, where $T\in\mathbf{R}^+$ satisfies the restriction presented in the theorem above.
\end{corollary}
\begin{proof}
Follows from T.\ref{contraction}.
\end{proof}

\subsection{Matrix Schr{\"o}dinger Unitary Groups}
If we consider that the space of states is discretizable, then we can compute a particular representation $\mathbb{H}_0$ of $H_0$, of approximation order $\nu_m$, that is exactly facotrizable since $H_0$ can be exactly factored, hence self adjoint, and will be the generator of a matrix unitary group $\{\hat{\mathbb{U}}:=e^{-it\mathbb{H}_0}\}$. It can be seen that the successive approximation method can be performed in two stages in order to compute an approximate discrete solution to \eqref{ivp} in time and space. First we can obtain a generic representation of the approximate expression of particular elements of the discrete Schr{\"o}dinger semigroup $\{U_k:=e^{-ikh\mathbb{H}_0},h\in\mathbf{R}^+,k\in\mathbf{Z}\}$, this can be done using a Picard type method in the following way. We start with the problem
\begin{equation}
i\ket{\Psi'(t)}=\mathbb{H}_0\ket{\Psi(t)}
\end{equation}
with $\ket{\Psi(0)}=\ket{\Psi_0}$, using the succesive approximation method with approximation operator defined explicitly by 
\begin{equation}
T(\Psi(t)):=\ket{\Psi_0}+\int_0^t\mathbb{H}_0\ket{\Psi(t)}
\end{equation}
we can obtain the approximatin sequence defined in D.\ref{approx_seq} explicitly given by
\begin{equation}
\ket{\Psi_n(h)}:=\mathbb{U}\ket{\Psi_0}=\sum_{k=0}^n \frac{1}{k!}(-i\tau\mathbb{H}_0)^k\ket{\Psi_0}
\label{disc_sem}
\end{equation}
where $\tau\in\mathbf{R}^+$ satisfies the Picard's restriction $\tau<\norm*{2}{\mathbb{W}_\mathbb{H}\mathbb{H}_0\mathbb{W}_\mathbb{H}^{-1}}^{-1}$. It can be seen that for these types of initial value problems the succesive approximation method produces a solution that coincides with the Taylor expansion of the exponential matrix $e^{-i\tau\mathbb{H}_0}$, we can take adavantage of this using the Pad{\'e} approximant of $e^{-i\tau\mathbb{H}_0}$ in the interval $[0,\tau]\subset \mathbf{R}$ for $\tau\in\mathbf{R}^+$ a basic time step, denoted by $\mathbb{U}\in M_{N_{m,h}}(\mathbf{C})$ and defined by $\mathbb{U}:=R_{pp}(-i\tau\mathbb{H}_0), p:=\left\lfloor n/2 \right\rfloor$ where $R_{pp}(-i\tau\mathbb{H}_0)$ is defined by:
\begin{equation}
R_{pp}(-i\tau\mathbb{H}_0):=D_{pp}(-i\tau\mathbb{H}_0)^{-1}N_{pp}(-i\tau\mathbb{H}_0)
\label{group_factors}
\end{equation}
with
\begin{eqnarray}
N_{pq}(-i\tau\mathbb{H}_0)&:=&\sum_{j=0}^p\frac{(p+q-j)!p!}{(p+q)!j!(p-j)!}(-i\tau\mathbb{H}_0)\\
D_{pq}(-i\tau\mathbb{H}_0)&:=&\sum_{j=0}^q\frac{(p+q-j)!q!}{(p+q)!j!(q-j)!}(i\tau\mathbb{H}_0)
\end{eqnarray}
It can be seen that taking $\mathbb{S}:=N_{pp}(-i\tau\mathbb{H}_0)$, we will have that $D_{pp}(-i\tau\mathbb{H}_0)=\mathbb{S}^\dagger$, and if we take $\mathbb{S}^+:=(\mathbb{S}^\dagger)^{-1}$, then we can express \eqref{group_factors} in the form
\begin{equation}
\mathbb{U}=\mathbb{S}^+\mathbb{S}
\end{equation}
From the relation of \eqref{group_factors} with the Taylor expansion and succesive approximation of $e^{-i\tau\mathbb{H}_0}$ and the Picard's resctriction for $\tau=h_\tau\norm*{\mathbb{H}}{\mathbb{H}_0}^{-1}<\norm*{\mathbb{H}}{\mathbb{H}_0}^{-1}$, with $0<h_\tau<1$ in \eqref{disc_sem}, we can obtain the following estimate
\begin{lemma} \label{lema_conv1}
$\norm*{2}{e^{-i\tau\mathbb{H}_0}-\mathbb{U}}\leq\left|\frac{1}{(2p+1)!}-c_{p,2p+1}\right|h_{\tau}^{2p+1}$
\end{lemma}
\begin{proof}
\begin{eqnarray}
\norm*{2}{e^{-i\tau\mathbb{H}_0}-\mathbb{U}}&\leq&\norm*{2}{\sum_{k=2p+1}^\infty (\frac{1}{k!}-c_{p,k})(-i\tau\mathbb{H}_0)^k}\\
&\leq&\left|\frac{1}{(2p+1)!}-c_{p,2p+1}\right||-\tau|^{2p+1}\norm*{2}{\mathbb{H}_0}^{2p+1}\\
&\leq&\left|\frac{1}{(2p+1)!}-c_{p,2p+1}\right|h_{\tau}^{2p+1}.
\end{eqnarray}
\end{proof}

Now if we want to compute the approximation of $e^{-it\mathbb{H}_0}$ corresponding to the interval $[0,t]\subset\mathbf{R}$, with $t:=m\tau, m\in\mathbf{Z}^+$, we will have that $e^{-it\mathbb{H}_0}=(e^{-i\tau\mathbb{H}_0})^m$ becomes approximated by $\mathbb{U}^m$, giving this an explicit definition for the discrete unitary group $\{\mathbb{U}_k:=\mathbb{U}^k, k\in\mathbf{Z}+$\}. Since $\mathbb{H}_0$ will be considered in general self adjoint, i.e., $\mathbb{H}_0^\dagger=\mathbb{H}_0$, we will have that $\mathbb{H}_0$ is normal, hence can be factored in the form $\mathbb{H}_0=\mathbb{V}\mathbb{D}_0\mathbb{V}^\ast$, with $\mathbb{D}_0:=\text{diag}\{d_i\}$, and taking $\Lambda_0:=R_{pp}(-i\tau\mathbb{D}_0)$ we obtain
\begin{eqnarray}
\mathbb{U}&=&\mathbb{S}^+\mathbb{S}\\
&=&\mathbb{V}\Lambda_0^+\mathbb{V}^\ast\mathbb{V}\Lambda_0\mathbb{V}^\ast\\
&=&\mathbb{V}\Lambda_0^+\Lambda_0\mathbb{V}^\ast
\end{eqnarray}
wich implies the following result.
\begin{lemma}
$\mathbb{U}^\ast\mathbb{U}=\mathbf{1}$.
\end{lemma}
\begin{proof}
\begin{eqnarray}
\mathbb{U}^\ast\mathbb{U}&=&\mathbb{V}\Lambda_0^\ast\Lambda_0^{-1}\mathbb{V}^\ast\mathbb{V}\Lambda_0^+\Lambda_0\mathbb{V}^\ast\\
&=&\mathbb{V}\Lambda_0^\ast\Lambda_0^{-1}\Lambda_0^+\Lambda_0\mathbb{V}^\ast\\
&=&\mathbb{V}\Lambda_0^\ast\Lambda_0^+\Lambda_0^{-1}\Lambda_0\mathbb{V}^\ast\\
&=&\mathbb{V}\mathbb{V}^\ast\\
&=&\mathbf{1}
\end{eqnarray}
\end{proof}
from the above relations can see that the operator
\begin{equation}
\hat{\mathbb{U}}:=\mathbb{W}_{\mathbb{H}}^{-1}\mathbb{U}\mathbb{W}_\mathbb{H}
\label{norm_ev}
\end{equation}
satisfies the following relations
\begin{equation}
\norm{}{\hat{\mathbb{U}}\Phi}=\norm*{2}{\mathbb{W}_\mathbb{H}\mathbb{W}_\mathbb{H}^{-1}\mathbb{U}\mathbb{W}_\mathbb{H}\Phi}=\norm*{2}{\mathbb{U}\mathbb{W}_\mathbb{H}\Phi}=\norm*{2}{\mathbb{W}_\mathbb{H}\Phi}=\norm{}{\Phi}
\end{equation}
wich implies that $\|\hat{\mathbb{U}}\|=1$, the adjoint of $\hat{\mathbb{U}}$ can be obtained in the following way:
\begin{eqnarray}
\hat{\mathbb{U}}^\dagger&=&\mathbb{M}_{\mathbb{H}}^{-1}\hat{\mathbb{U}}^\ast\mathbb{M}_{\mathbb{H}}\\
&=&\mathbb{M}_{\mathbb{H}}^{-1}(\mathbb{W}_{\mathbb{H}}^{-1}\mathbb{U}\mathbb{W}_\mathbb{H})^\ast\mathbb{M}_{\mathbb{H}}\\
&=&\mathbb{M}_{\mathbb{H}}^{-1}\mathbb{W}_{\mathbb{H}}^\ast\mathbb{U}^\ast\mathbb{W}_\mathbb{H}^+\mathbb{M}_{\mathbb{H}}\\
&=&\mathbb{W}_{\mathbb{H}}^{-1}\mathbb{W}_{\mathbb{H}}^{+}\mathbb{W}_{\mathbb{H}}^\ast\mathbb{U}^\ast\mathbb{W}_\mathbb{H}^+\mathbb{W}_{\mathbb{H}}^\ast\mathbb{W}_{\mathbb{H}}\\
&=&\mathbb{W}_{\mathbb{H}}^{-1}\mathbb{U}^\ast\mathbb{W}_{\mathbb{H}}
\end{eqnarray}
the discrete Schr{\"o}dinger unitary group relative to $\mathbb{H}_0$, will have the form $\{\hat{\mathbb{U}}_k:=\hat{\mathbb{U}}^k(\cdot),k\in\mathbf{Z}^+\}$, wich is consistent with the normalization presented at \textsection\ref{qds}. It can be seen that
\begin{equation}
\hat{\mathbb{U}}^\dagger\hat{\mathbb{U}}=\mathbf{1}=\hat{\mathbb{U}}\hat{\mathbb{U}}^\dagger
\end{equation}
and this implies that the discrete time reversal Schr{\"o}dinger unitary group will be given by $\{\hat{\mathbb{U}}_{-k}:=(\hat{\mathbb{U}}^\dagger)^{k}(\cdot),k\in\mathbf{Z}^+\}$ and will be coherent with the local time reversibility of Schr{\"o}dinger unitary groups. From lemma L.\ref{lema_conv1} we can obtain the following.

\begin{lemma} \label{lema_conv2} 
$\norm{}{e^{-im\tau\mathbb{H}_0}-\hat{\mathbb{U}}^m}\leq\frac{{m}^{2p+1}}{(2p+1)!}h_{\tau}^{2p+1}$.
\end{lemma}
\begin{proof}
If we denote by $c_{p,k}$ the multinomial Pad{\'e} coefficients, then we will get
\begin{eqnarray}
\norm{}{e^{-im\tau\mathbb{H}_0}-\hat{\mathbb{U}}^m}&=&\norm{}{\sum_{k=0}^\infty \frac{(-im\tau\mathbb{H}_0)^k}{k!}-(\sum_{k=0}^\infty c_{p,k}(-i\tau\mathbb{H}_0)^k)^m}\\
&\leq&\norm{}{\sum_{k=2p+1}^\infty (\frac{m^k}{k!}-c_{p,k})(-i\tau\mathbb{H}_0)^k}\\
&\leq& \left|\frac{m^{2p+1}}{(2p+1)!}-c_{p,2p+1}\right| |-\tau|^{2p+1}\norm{}{\mathbb{H}_0}^{2p+1}\\
&\leq&\frac{m^{2p+1}}{(2p+1)!}h_\tau^{2p+1}.
\end{eqnarray}
\end{proof}

If for a given self-adjoint operator $H_1\in\mathscr{L}(\mathcal{H})$ we compute its representation with respect to a particular projection $P_\mathcal{S}$ denoted by $\mathbb{H}_1$, then we can obtain the following convergence result.

\begin{lemma} \label{lema_conv3}
$\norm{}{e^{-im\tau H_1}-\hat{\mathbb{U}}^m}\leq c_1 h_x^{\nu_m} + c_2 h_\tau^{2p+1}$.
\end{lemma}
\begin{proof}
By the properties of particular projections we get that
\begin{eqnarray}
\norm{}{e^{-im\tau H_1}-\hat{\mathbb{U}}^m}&\leq&\norm{}{e^{-im\tau H_1}-e^{-im\tau\mathbb{H}_1}}+\norm{}{e^{-im\tau\mathbb{H}_1}-\hat{\mathbb{U}}^m}\\
&\leq& c_{H_1} h^{\nu_m}+\frac{m^{2p+1}}{(2p+1)!}h_\tau^{2p+1}
\end{eqnarray}
and taking $c_1:=c_{H_1}$ and $c_2:=\frac{m^{2p+1}}{(2p+1)!}$ the result follows.
\end{proof}

Now if for a particular representation $\mathbb{H}_0$ of $H_0$ we denote by $\{\hat{\mathbb{U}}_{k},k\in\mathbf{Z}^+\}$ and by $\{\hat{\mathbb{U}}_{-k},k\in\mathbf{Z}^+\}$ the direct and reverse Schr{\"o}dinger semigroups respectively, we can obtain a fully discrete representation for \eqref{exp_contract} in the form
\begin{equation}
\hat{\mathbb{S}}_k(\cdot):=\hat{\mathbb{U}}_k\ket{\Psi_0}+\sum_{j=0}^k w_j^k(\tau)\hat{\mathbb{U}}_{k-j}V(\cdot)
\label{exp_contract_disc}
\end{equation}
in this expression $\{w_j^k(\tau)\}$ are chosen such that the integration in time is exact for the polynomial structure of the Taylor similar part of $\hat{\mathbb{U}}$, wich allows $\hat{\mathbb{S}}_k$ to mimic the operator defined in \eqref{exp_contract}. Using the properties of particula projection methods we can derive the following discrete results.

\begin{theorem}
The operator $\hat{\mathbb{S}}_k$ is a strict contraction with respect to the norm $\norm*{C([-\tau,\tau],\mathcal{D}(G))}{\cdot}$, when $V\in C^{\alpha=1}(\mathcal{D}\subset \mathcal{H})$ and $h<1/c_V.$
\end{theorem}
\begin{proof}
First it is important to note, that since the quadrature rule defined by the sequence $\{w_j^k(\tau)\}$ is exact for the polynomial structure of $\hat{\mathbb{U}}$ and if we take $\Psi^j:=\Psi(j\tau/k)$ for any given function $\Psi\in C([-\tau,\tau],\mathcal{H}(G))$, we will have that
\begin{eqnarray*}
\norm*{\mathbb{H}}{\hat{\mathbb{S}}_k(\Phi(t))-\hat{\mathbb{S}}_k(\Upsilon(t))}
&=&\norm*{\mathbb{H}}{\sum_{j=0}^k w_j^k(\tau)\hat{\mathbb{U}}_{k-j}(V(\Phi^j)-V(\Upsilon^j))}\\
&\leq&\sum_{j=0}^k w_j^k(\tau)\norm*{\mathbb{H}}{\hat{\mathbb{U}}_{k-j}}\norm*{\mathbb{H}}{V(\Phi^j)-V(\Upsilon^j)}\\
&\leq&\sum_{j=0}^k w_j^k(\tau)c_V\norm*{\mathbb{H}}{\Phi^j-\Upsilon^j}
\end{eqnarray*}
the last relation implies that
\begin{eqnarray*}
\norm*{C([-\tau,\tau],\mathcal{D}(G))}{\hat{\mathbb{S}}_k(\Phi(t))-\hat{\mathbb{S}}_k(\Upsilon(t))}&=&\sup_{t\in[-\tau,\tau]}\norm*{\mathbb{H}}{\hat{\mathbb{S}}_k(\Phi(t))-\hat{\mathbb{S}}_k(\Upsilon(t))}\\
&\leq&\left(\sum_{j=0}^k w_j^k(\tau) \right)c_V\sup_j\norm*{\mathbb{H}}{\Phi^j-\Upsilon^j}\\
&\leq&c_V\tau\norm*{C([-\tau,\tau],\mathcal{D}(G))}{\Phi-\Upsilon}
\end{eqnarray*}
and taking $\tau<1/c_V$ the result follows.
\end{proof}

\begin{corollary} \label{disc_exist} The semidiscrete Schr{\"o}dinger initial value problem
\begin{equation}
\left\{
\begin{array}{l}
i\ket{\Psi'(t)}=\mathbb{H}_0\ket{\Psi(t)}+V(\ket{\Psi(t)}) \\
\ket{\Psi(0)}=\Psi_0
\end{array}
\right .
\label{divp}
\end{equation}
has a unique solution determined by $\ket{\Psi_0}$ on an interval $[-\tau,\tau]$ for $\tau<1/c_V$.
\end{corollary}
\begin{proof}
Follows from T.\ref{contraction}.
\end{proof}
\begin{lemma}
A solution $\ket{\Psi(t)}$ to \eqref{divp} satisfy the following estimate with respect to the its appoximating sequence
\begin{equation}
\norm*{C([-\tau,\tau],\mathcal{D}(G))}{\ket{\Psi(t)}-\ket{\Psi_n(t)}}\leq\frac{(c_V\tau)^n}{1-(c_V\tau)}(2\norm*{C([-\tau,\tau],\mathcal{D}(G))}{\ket{\Psi_0}}+M_V\tau)
\end{equation}
where $M_V:=\sup_{x}\|V(x)\|_\mathbb{H}$.
\end{lemma}
\begin{proof}
It can be seen that 
\begin{eqnarray}
\norm*{\mathbb{H}}{\ket{\Psi_1(\tau)}-\ket{\Psi_0(\tau)}}&=&\norm*{\mathbb{H}}{\hat{\mathbb{S}}(\ket{\Psi_0})-\ket{\Psi_0}}\\
&=&\norm*{\mathbb{H}}{\hat{\mathbb{U}}_k\ket{\Psi_0}+\sum_{j=0}^k w_j(k)\hat{\mathbb{U}}_{k-j}V(\ket{\Psi_0})-\ket{\Psi_0}}\\
&\leq&2\norm*{\mathbb{H}}{\ket{\Psi_0}}+\sum_{j=0}^kw_j(k)\norm*{\mathbb{H}}{\hat{\mathbb{U}}_{k-j}}\norm*{\mathbb{H}}{V(\ket{\Psi_0})}\\
&\leq&2\norm*{\mathbb{H}}{\ket{\Psi_0}}+\tau\norm*{\mathbb{H}}{V(\ket{\Psi_0})}
\end{eqnarray}
also it can be seen that $\sup_\tau\norm*{\mathbb{H}}{\ket{\Psi_1(\tau)}-\ket{\Psi_0(\tau)}}\leq2\norm*{C([-\tau,\tau],\mathcal{D}(G))}{\ket{\Psi_0}}+M_V\tau$, replacing above expressions in R.\ref{rcontraction} and taking $\tau<1/c_V$ and $K=c_V\tau$ the result follows.
\end{proof}

\begin{theorem} \label{estimate_disc}
If we denote by $\ket{\psi(t)}\in\mathcal{H}$ and by $\ket{\Psi_n(t)}\in\mathbb{H}$ the exact solution to \eqref{ivp} and the n-th time approximating solution to \eqref{divp} respectively and if we have that particular representations of quantum operators in $\mathbb{H}$ have approximation order $\nu_m$. Then we will have that $\ket{\psi(t)}\in\mathcal{H}$ and $\ket{\Psi_n(t)}\in\mathbb{H}$ satisfy the following estimate
\begin{equation}
\norm*{C([-\tau,\tau],\mathcal{D}(G))}{\ket{\psi(t)}-\ket{\Psi_n(t)}}\leq c_\psi h^{\nu_m}+\frac{(c_V\tau)^n}{1-(c_V\tau)}(2\norm*{C([-\tau,\tau],\mathcal{D}(G))}{\ket{\Psi_0}}+M_V\tau)
\end{equation}
\end{theorem}
\begin{proof}
If we denote by $\ket{\Psi(t)}\in\mathbb{H}$ the exact solution to \eqref{divp} then we will have that, since particular representations of quantum operators in $\mathbb{H}$ have approximation order $\nu_m$, we will have that $\norm*{\mathbb{H}}{\ket{\psi(t)}-\ket{\Psi(t)}}\leq C_\psi(t)h^{\nu_m}, \forall t\in[-\tau,\tau],0<\tau<\infty$, taking $c_\psi:=\sup_t C_\psi(t)$, we will obtain the relation $\norm*{C([-\tau,\tau],\mathcal{D}(G))}{\ket{\psi(t)}-\ket{\Psi(t)}}\leq c_\psi h^{\nu_m}$. Hence
\begin{eqnarray*}
\norm*{C([-\tau,\tau],\mathcal{D}(G))}{\ket{\psi(t)}-\ket{\Psi_n(t)}}&\leq&\norm*{C([-\tau,\tau],\mathcal{D}(G))}{\ket{\psi(t)}-\ket{\Psi(t)}}\\
&&+\norm*{C([-\tau,\tau],\mathcal{D}(G))}{\ket{\Psi(t)}-\ket{\Psi_n(t)}}\\
&\leq&c_\psi h^{\nu_m}+\frac{(c_V\tau)^n}{1-(c_V\tau)}(2\norm*{C([-\tau,\tau],\mathcal{D}(G))}{\ket{\Psi_0}}\\
&&+M_V\tau).
\end{eqnarray*}
wich provides the desired result.
\end{proof}

\begin{remark} \label{obs2}
Remark R.\ref{obs1} implies that for any two compatible operators $A,B\in\mathscr{L}(\mathcal{H})$ we will have that
\begin{eqnarray}
\mathbb{O}=P_\mathcal{S}\mathbf{0}P_\mathcal{S}&=&P_\mathcal{S}[A,B]P_\mathcal{S}\\
&=&P_\mathcal{S}(AB-BA)P_\mathcal{S}\\
&=&P_\mathcal{S}ABP_\mathcal{S}-P_\mathcal{S}BAP_\mathcal{S}\\
&=&P_\mathcal{S}AP_\mathcal{S}BP_\mathcal{S}-P_\mathcal{S}BP_\mathcal{S}AP_\mathcal{S}\\
&=&P_\mathcal{S}AP_\mathcal{S}^2BP_\mathcal{S}-P_\mathcal{S}BP_\mathcal{S}^2AP_\mathcal{S}\\
&=&A_\mathcal{S}B_\mathcal{S}-B_\mathcal{S}A_\mathcal{S}\\
&=&[A_\mathcal{S},B_\mathcal{S}].
\end{eqnarray}
\end{remark}

\begin{corollary} Discrete constants of motion. A particular projection $\mathbb{A}\in\mathscr{L}(\mathbb{H})$ of an operator $A\in\mathscr{L}(\mathcal{H})$ that is compatible with $H_0\in\mathscr{L}(\mathcal{H})$ defines a constant of motion $\braket{\mathbb{A}}_t$
for the system modeled by:
\begin{equation}
\left\{
\begin{array}{l}
i\partial_t\ket{\Psi(t)}=\mathbb{H}_0\ket{\Psi(t)}\\
\ket{\Psi(0)}=\Psi_0
\end{array}
\right . 
\end{equation}
where $\mathbb{H}_0:=p^\dagger_\mathcal{S}H_0p_\mathcal{S}$ and $\mathbb{H}_0:=p^\dagger_\mathcal{S}H_0p_\mathcal{S}$.
\end{corollary}
\begin{proof}
From R.\ref{obs2} we get that $[\mathbb{A},\mathbb{H}_0]=\mathbb{O}$, then the result follows from T.\ref{const_motion} 
\end{proof}

\begin{remark}
It is important to observe that $[A_{m,h},B_{m,h}]=\mathbb{O}, \forall h\in \mathbf{R}+$, then we will have that $[A,B]=\mathbf{0}$, with $A:=\lim_{h\to0^+} A_{m,h}$ and $B:=\lim_{h\to0^+}B_{m,h}$.
\end{remark}

\section{Examples}

In this section we will present some examples of particular implementation of approximation schemes and some related estimates.

\begin{example} Quantum Harmonic Oscillator. Given a quantum system on a square box $G=[0,L]\times[0,L], L<\infty$ whose time evolution is described by a quantum harmonic oscillator equation that has the form
\begin{equation}
i\frac{\partial}{\partial t}\ket{\psi(t)}:=\Delta_2\ket{\psi(t)} +(x^2+y^2)\ket{\psi(t)}
\label{exa1}
\end{equation}
with $\psi(0)=\psi_0\in\mathcal{S}(G)\subset \mathcal{H}(G)$and if the model is restricted to physical conditions that prevent any representative particle to be located an the boundary, i.e., $|\ket{\psi(t)}|^2=0,x\in\partial G, t\in\mathbf{R}$. Physical restrictions imply null Dirichlet boundary conditions for \eqref{exa1}. If we take $H_0:=\Delta_2:=\nabla\cdot\nabla=p^\dagger p$, wich is clearly exactly factorizable, and if we take its particular projection to $\mathbb{H}$ a spectral element space of first order, then we can express \eqref{exa1} in the form
\begin{equation}
i\frac{d}{dt}\ket{\Psi(t)}:=\mathbb{M}^{-1}\mathbb{P}^\ast\mathbb{M}'\mathbb{P}\ket{\Psi(t)} +(\mathbb{X}^2+\mathbb{Y}^2)\ket{\Psi(t)}
\label{exa1_disc}
\end{equation}
\end{example}
with $\ket{\Psi(0)}=\Psi_0$. It can be seen that $V(\cdot)\in C^{\alpha=1}(\mathcal{D}(G)\subset\mathcal{H}(G))$ is linear and that $c_V:=2L^2$. Using Cea's lemma and a procedure similar to the followed in Chapter VII \textsection 5 in \cite{Showalter}, it can be seen that the particular representation in \eqref{exa1_disc} has approximation order $\nu_m=2$. Also from \eqref{weak_factors_disc} we can observe that $\mathbb{H}_1:=\mathbb{H}_0+\mathbb{X}^2+\mathbb{Y}^2$ is self-adjoint,
hence $i\mathbb{H}_1$ is conservative and $e^{\pm it\mathbb{H}_1}$ will be unitary with respecto to $\norm*{\mathbb{H}}{\cdot}$. If we take $\hat{\mathbb{U}}$ to be given by the Crank-Nicholson approximation $\hat{\mathbb{U}}:=(\mathbf{1}+i\tau\mathbb{H}_1)^{-1}(\mathbf{1}-i\tau\mathbb{H}_1)$ wich is clearly a Pad{\'e} approximant of $e^{-it\mathbb{H}_1}$ and if we replace this approximation in \eqref{exp_contract_disc}, since all the conditions are satisfied, then by C.\ref{disc_exist} and L.\ref{lema_conv3} we will have that there exists an approximate solution $\ket{\Psi_1(t)}$ to \eqref{exa1_disc} that converges to the exact solution to \eqref{exa1} according to the estimate
\begin{eqnarray}
\norm*{}{\ket{\psi(t)}-\ket{\Psi_1(t)}}&\leq& c_\psi h^2+\frac{m^3}{6}h_\tau^3
\end{eqnarray}
where $h\in\mathbf{R}^+$ is the mesh size of the particular grid $G_{1,h}\subset G$ implemented. It is also important to note that if we denote the energy expectation evolution by $\braket{E}_t$ and its discrete representation by $\braket{\mathbb{E}}_{n}$ then it can be seen that
\begin{eqnarray}
\delta_\tau\braket{\mathbb{E}}_{n}&=&\frac{1}{\tau}(\braket{\mathbb{E}}_{n+1}-\braket{\mathbb{E}}_{n})\\
&=&\frac{1}{\tau}(\bra{\hat{\mathbb{U}}^{n+1}\Psi_0}\mathbb{H}_1\ket{\hat{\mathbb{U}}^{n+1}\Psi_0}-\bra{\hat{\mathbb{U}}^n\Psi_0}\mathbb{H}_1\ket{\hat{\mathbb{U}}^n\Psi_0})\\
&=&\bra{\frac{1}{\tau}(\hat{\mathbb{U}}-\mathbf{1})\hat{\mathbb{U}}^n\Psi_0}\mathbb{H}_1\ket{\hat{\mathbb{U}}^n\Psi_0}\nonumber\\
&&+\bra{\hat{\mathbb{U}}^n\Psi_0}\mathbb{H}_1\ket{\frac{1}{\tau}(\hat{\mathbb{U}}-\mathbf{1})\hat{\mathbb{U}}^n\Psi_0} \nonumber\\
&&+\bra{(\hat{\mathbb{U}}-\mathbf{1})\hat{\mathbb{U}}^n\Psi_0}\mathbb{H}_1\ket{\frac{1}{\tau}(\hat{\mathbb{U}}-\mathbf{1})\hat{\mathbb{U}}^n\Psi_0}
\end{eqnarray}
now it can be seen that the total variation for $\tau=0$ gives 
\begin{equation}
\delta_0\braket{\mathbb{E}}_n=i\bra{\hat{\mathbb{U}}^n\Psi_0}[\mathbb{H}_1,\mathbb{H}_1]\ket{\hat{\mathbb{U}}^n\Psi_0}=0
\end{equation}
wich means that $\delta_0$ kills $\braket{\mathbb{E}}_n$, $\forall n$, i.e., $\braket{\mathbb{E}}_n$ is a discrete constant of motion.

\begin{example} Nonlinear Schr{\"o}dinger Equation. For quantum dynamical systems described by a nonlinear Schr{\"o}dinger equation of the form
\begin{equation}
i\frac{\partial}{\partial t}\ket{\psi(t)}:=\Delta_2\ket{\psi(t)} +\alpha|\psi(t)|^n\ket{\psi(t)}
\label{exa2}
\end{equation}
with $\ket{\psi(0)}:=\ket{\psi_0}$, $2\leq n\in\mathbf{Z}^+$ and $\alpha\in\mathbf{C}$. If we have that the quantum dynamical systems in this example evolve in the same media described in example I and have the same physical localization restrictions for its representative particles. Then the only condition that we need to check is the locally Lipschitz condition of $V(X):=\alpha|X|^n X$, so we will need to restrict our attention to functions on a set $\mathcal{S}(G)\subset\mathcal{H}$ that is caracterized by $\mathcal{S}(G):=\{\phi\in\mathcal{H}(G): \norm{}{\phi}\leq r, \mathbf{R}^+\ni r<\infty\}$, now, since $||X|^nX-|Y|^nY|\leq C_n(|X|^n+|Y|^n)|X-Y|$, where $C_n>0$ is a value that depends on $n$, we will have that $c_V:=2|\alpha| C_nr^n$ is the Lipschitz constant of $V(\cdot)$ with respect to $\norm*{\mathcal{H}(G)}{\cdot}$. If we use the same spatial particular projection than in the first example reewriting \eqref{exa2} in the form
\begin{equation}
i\frac{d}{dt}\ket{\Psi(t)}:=\mathbb{M}^{-1}\mathbb{P}^\ast\mathbb{M}'\mathbb{P}\ket{\Psi(t)} +\alpha|\Psi(t)|^n\ket{\Psi(t)}
\label{exa2_disc}
\end{equation}
and if we also approximate $\mathbb{U}$ using the Crank-Nicholson approximation then by C.\ref{disc_exist} and T.\ref{estimate_disc} we will have that there exists an approximate solution $\ket{\Psi_1(t)}$ to problems of the form \eqref{exa2_disc} that converges to the exact solution to \eqref{exa1} according to the estimate
\begin{equation}
\norm*{C([-\tau,\tau],\mathcal{D}(G))}{\ket{\psi(t)}-\ket{\Psi_1(t)}}\leq c_\psi h^2+\frac{2|\alpha| C_nr^{n+1}\tau}{1-2|\alpha| C_nr^n\tau}(2 + |\alpha| r^n\tau)
\end{equation}
where $h\in\mathbf{R}^+$ is the mesh size of the particular grid $G_{1,h}\subset G$ implemented.
\end{example}

In a similar way we can implement the operator techniques used here to study convergence and other related properties for other particular projectors with different approximation orders, also other Pad{\'e} approximants than the Crank-Nicholson scheme can be used for the approximation of the integrating factor.

\section{Conclusion}
Particular projection methods can be implemented in the functional numerical analysis of several types of Schr{\"o}dinger evolution equations, in this work we introduce some of the strategies that can be used and also we derive some estimates that can be very useful when we deal with the numerical solution of such important equations in quantum physics.

\begin{center}
 \normalsize\scshape{Acknowledgements}
\end{center}

\medskip

I want to say Thanks: To Hashem for everything, to Mirna, for her love, for her support, for making me laugh... to Stanly Steinberg for his support and for so many useful comments and suggestions, to Concepci{\'o}n Ferrufino, Rosibel Pacheco and Jorge Destephen for their support and advice.

\appendix

\section{Some Results from Functional Analysis}

In this section we present some important results from functional analysis that are used in this work.

\begin{definition} An operator $T\in\mathcal{B}(X)$ with $X$ a Banach space, for wich 
\begin{equation}
\norm{}{T(x)-T(y)}\leq\norm{}{x-y}, x,y\in X 
\end{equation}
is called a contraction. If there is a $K<1$ for wich $\norm{}{T(x)-T(y)}\leq K\norm{}{x-y}$, T is called a strict contraction.
\end{definition}
\begin{theorem}\label{contraction}Contraction Mapping Principle. A strict contraction $T\in\mathcal{B}(X)$ on a Banach space $X$ has a unique fixed point, ie., there exists a unique $x\in X$ such that $T(x)=x$.
\end{theorem}
\begin{definition} \label{approx_seq} Succesive approximation methods for fixed points. Given a strict contraction $T\in\mathscr{L}(X)$ on a banach space $X$, the sequence $\{x_n\}$ defined for some $y\in X$ in the form
\begin{equation}
x_n:=\left \{
\begin{array}{l}
y, n=0\\
T(x_{n-1}), n\geq 1
\end{array}
\right .
\end{equation}
will be called an approximation sequence to $x$, and the iterative method will receive the name of successive approximations method.
\end{definition}
\begin{remark}\label{rcontraction}
A fixed point $x\in X$ of a contraction $T\in\mathcal{B}(X)$ in a Banach space $X$, satisfies the following estimate
\begin{equation}
\norm{}{x-x_m}\leq K^m(1-K)^{-1}\norm{}{x_1-x_0}.
\end{equation}
with respect to its approximating sequence.
\end{remark}
\begin{definition} Locally Lipschitz Functions. A function $f$ is locally Lipschitz ,i.e., $f\in C^{\alpha=1}(\overline{B}_r(0))$, with $\overline{B}_r(0)$ a closed ball of radius $r$ centered in $0$, if for $0<\varepsilon_r\in\mathbb{R}$ that depends on $r$, small enough, there exists $c_f<\infty$ with
\begin{equation}
\norm*{H^m(G)}{f(u)-f(v)}\leq c_f\norm*{H^m(G)}{u-v}
\end{equation}
when $\norm*{H^m(G)}{u-v}\leq\varepsilon_r$.
\end{definition}

\begin{definition} \label{disip} Dissipative Operator. An elliptic operator $A\in\mathscr{L}(H^n(G))$ is said to be dissipative if we have that
\begin{equation}
\text{Re} \scalprod{}{Ax}{x}\leq 0 \:\:, x\in \text{dom}(A).
\end{equation}
in particular if equality holds, then $A$ is said to be conservative.
\end{definition}

\begin{theorem}
If $A\in\mathcal{L}(\mathcal{H}(G))$ where $\mathcal{H}(G)$ is a discretizable Hilbert space, and if $A$ is closed, densely defined and dissipative then it generates a contractive semigroup.
\end{theorem}

\begin{theorem} \label{stone} Stone's Theorem. A densely defined operator $iA \in \mathscr{L}(\mathcal{H})$ in a complex Hilbert space $\mathcal{H}$ is the generator of a strongly continuous unitary group on H if and only if A is self-adjoint.

\end{theorem}

\begin{thebibliography}{9}
\bibitem{Trotter} \textsc{Trotter H.F.}: \textit{Approximation of Semigroups of Operators}, Transactions of the American Mathematical Society, 1957.
\bibitem{ReedSimon}\textsc{Reed M. and Simon B.}: \textit{Methods of Modern Mathematical Physics: Functional Analysis.} Academic Press, Inc. New York, 1972.
\bibitem{Steinberg1980} \textsc{Steinberg S.}: \textit{Local Propagator Theory}, Rocky Mountain Journal of Mathematics, Volume 10, Number 4, Fall 1980.
\bibitem{Simon} \textsc{Simon B.}: \textit{Schr{\"o}dinger Semigroups}, Bulletin (New Series) of the American Mathematical Society, Volume 7, Number 3, November 1982.
\bibitem {Showalter} \textsc{Showalter, R. E.}: \textit{Hilbert Space Methods for
Partial Differential Equations}, Electronic Journal of Differential Equations
Monograph 01, 1994.
\bibitem{Hunter} \textsc{Hunter J.K.}: \textit{Nonlinear Evolution Equations}, University of California, Davis, 1996.
\bibitem{Boyd} \textsc{Boyd J.P.}: \textit{Chebyshev and Fourier Spectral Methods}, Second Edition, DOVER Publications, Inc. 2000.
\bibitem{Chebotarev} \textsc{Chebotarev A.M.}: \textit{Lecture on Quantum Probability}, Aportaciones Matem{\'a}ticas, Sociedad Matem{\'a}tica Mexicana, S y G Editores, S.A. de C.V., 2000.
\bibitem{Bottcher} \textsc{B{\"o}ttcher A.}: \textit{C*-Algebras in Numerical Analysis}, Irish Math. Soc. Bulletin 45 (2000), 57–133, 2000.
\bibitem{Moler} \textsc{Moler C. and Van Loan C.}: \textit{Nineteen Dubious Ways to Compute the Exponential of a Matrix, Twenty-Five Years Later}, SIAM REVIEW Vol. 45, No. 1, pp. 3–000, 2003.
\bibitem{Phillips} \textsc{Phillips A.C.}: \textit{Introduction to Quantum Mechanics}, The Manchester Physics Series, John Wiley \& Sons Ltd, 2003.
\bibitem {Steinberg2004} \textsc{Steinberg S.}: \textit{A Discrete Calculus with Applications of High-Order Discretizations to Boundary-Value Problems}, CMAM (Computational Methods of Applied Mathematics), \textbf{42},  4 (2004), 228-261. 2004.
\bibitem{Loring}\textsc{Loring T.}: \textit{From Matrix to Operator Inequalities}, arXiv:0902.0102v1 [math.OA], Department of Mathematics and Statistics, University of New Mexico, Albuquerque, NM 87131, USA. 2009.
\bibitem{Vides}\textsc{Vides F.}: \textit{On the Approximation of Contractive Semigroups of Operators in Discretizable Hilbert Spaces}, arXiv:1012.5106v1 [math.NA]. 2010.
\bibitem{Vides2}\textsc{Vides F.}: \textit{On the Approximation of Nonlinear Abstract Evolution Equations in Particular C*-Algebras of Operators}, arXiv:1012.5103v2 [math.FA]. 2010.
\end{thebibliography}
\end{document}